\RequirePackage{ifpdf} 
\ifpdf
\documentclass[11pt,pdftex]{article}
\else
\documentclass[11pt,dvips]{article}
\fi

\ifpdf
\usepackage{lmodern}
\fi
\usepackage[bookmarks, 
colorlinks=false, backref,plainpages=false]{hyperref}
\usepackage[T1]{fontenc}
\usepackage[latin1]{inputenc}
\usepackage[english]{babel}
\usepackage{graphicx}
\usepackage{url}
\usepackage{graphics}
\usepackage{epsfig,color}
\usepackage{amsmath,amssymb}
\usepackage{amsthm}
\renewcommand{\theequation}{\thesection.\arabic{equation}}

\newtheorem{thm}{Theorem}[section]

\newtheorem{lem}[thm]{Lemma}

\newtheorem{rem}[thm]{Remark}

\begin{document}
\newcommand{\BX}{{\bf X}}
\newcommand{\cv}{{\cal V}}
\newcommand{\cW}{{\cal W}}
\newcommand{\co}{{\cal O}}

\renewcommand{\theequation}{\thesection.\arabic{equation}}
\def\@eqnnum{{\reset@font\rm (\theequation)}}

\def\abstract{
\advance \rightskip by 10mm
\advance \leftskip by 10mm
\vspace{-0.8em}
\noindent
\small{\bf Abstract.}
}
\def\endabstract{\par\normalsize\rm}

\def\Xint#1{\mathchoice
{\XXint\displaystyle\textstyle{#1}}%
{\XXint\textstyle\scriptstyle{#1}}%
{\XXint\scriptstyle\scriptscriptstyle{#1}}%
{\XXint\scriptscriptstyle\scriptscriptstyle{#1}}%
\!\int}
\def\XXint#1#2#3{{\setbox0=\hbox{$#1{#2#3}{\int}$}
\vcenter{\hbox{$#2#3$}}\kern-.5\wd0}}
\def\ddashint{\Xint=}
\def\dashint{\Xint-}

\def\a{\alpha}
\def\b{\beta}
\def\d{\delta}\def\D{\Delta}
\def\e{\epsilon}
\def\g{\gamma}\def\G{\Gamma}
\def\k{\kappa}
\def\lam{\lambda}\def\Lam{\Lambda}
\renewcommand\o{\omega}\renewcommand\O{\Omega}
\def\s{\sigma}\def\S{\Sigma}
\renewcommand\t{\theta}\def\vt{\vartheta}
\newcommand{\vphi}{\varphi}
\def\z{\zeta}

\newcommand{\tsigma}{\tilde{\s}}
\newcommand{\tbsigma}{\tilde{\bsigma}}
\def\te{\tilde{\e}}
\def\tu{\tilde{u}}

\newcommand{\bchi}{\mbox{\boldmath$\chi$}}
\newcommand{\bdelta}{\mbox{\boldmath$\delta$}}
\newcommand{\bepsilon}{\mbox{\boldmath$\epsilon$}}
\newcommand{\bfeta}{\mbox{\boldmath$\eta$}}
\newcommand{\bgamma}{\mbox{\boldmath$\gamma$}}
\newcommand{\bomega}{\mbox{\boldmath$\omega$}}
\newcommand{\bvphi}{\mbox{\boldmath$\varphi$}}
\newcommand{\bphi}{\mbox{\boldmath$\phi$}}
\newcommand{\bPhi}{\mbox{\boldmath$\Phi$}}
\newcommand{\bpsi}{\mbox{\boldmath$\psi$}}
\newcommand{\bPsi}{\mbox{\boldmath$\Psi$}}
\newcommand{\bsigma}{\mbox{\boldmath$\sigma$}}
\newcommand{\btau}{\mbox{\boldmath$\tau$}}
\newcommand{\bxi}{\mbox{\boldmath$\xi$}}
\newcommand{\brho}{\mbox{\boldmath$\rho$}}
\newcommand{\bbeta}{\mbox{\boldmath$\beta$}}
\newcommand{\bzeta}{\mbox{\boldmath$\zeta$}}

\def\bk{\boldsymbol{\kappa}}
\def\bmu{\boldsymbol\mu}
\def\bxi{\boldsymbol{\xi}}
\def\bz{\boldsymbol{\zeta}}

\def\ba{{\bf a}}
\def\bb{{\bf b}}
\def\bc{{\bf c}}
\def\be{{\bf e}}
\def\bff{{\bf f}}
\def\bg{{\bf g}}
\def\bn{{\bf n}}
\def\bp{{\bf p}}
\def\bq{{\bf q}}
\def\bs{{\bf s}}
\def\bt{{\bf t}}
\def\bu{{\bf u}}
\def\bv{{\bf v}}
\def\bw{{\bf w}}
\def\bx{{\bf x}}
\def\by{{\bf y}}
\def\bzz{{\bf z}}

\def\bD{{\bf D}}
\def\bE{{\bf E}}
\def\bF{{\bf F}}
\def\bH{{\bf H}}
\def\bJ{{\bf J}}
\def\bV{{\bf V}}
\def\bU{{\bf U}}
\def\bW{{\bf W}}
\def\bX{{\bf X}}
\def\bY{{\bf Y}}

\def\cA{{\cal A}}
\def\cC{{\cal C}}
\def\cD{{\cal D}}
\def\cE{{\cal E}}
\def\cF{{\cal F}}
\def\cG{{\cal G}}
\def\cI{{\cal I}}
\def\cJ{{\cal J}}
\def\cK{{\cal K}}
\def\cL{{\cal L}}
\def\cO{{\cal O}}
\def\cP{{\cal P}}
\def\cQ{{\cal Q}}
\def\cR{{\cal R}}
\def\cS{{\cal \Sigma}}
\def\cT{{\cal T}}
\def\cU{{\cal U}}
\def\cV{{\cal V}}

\def\scT{{_\cT}}
\def\sD{{_D}}
\def\sE{{_E}}
\def\sF{{_F}}
\def\sFz{{_{F_z}}}
\def\sK{{_K}}
\def\sI{{_I}}
\def\sb{{_b}}
\def\sN{{_N}}

\def\curl{{{\bf curl} \ }}
\def\rot{{\mbox{rot}\ }}
\def\BPI{{\bf \Pi}}

\def\cth{\cT_h}
\def\ctH{\cT_H}

\def\tJ{\tilde{\J}}

\def\hK{\widehat{K}}
\def\hx{\widehat{x}}
\def\hy{\widehat{y}}
\def\bhv{\widehat{\bv}}

\def\l{\ell}
\def\bl{\boldsymbol{\ell}}
\def\col{\colon}
\def\f12{\frac12}
\def\dfrac{\displaystyle\frac}
\def\dint{\displaystyle\int}
\def\nab{\nabla}
\def\p{\partial}
\def\sm{\setminus}
\def\dsum{\displaystyle\sum}
\newcommand{\pp}[2]{\frac{\partial {#1}}{\partial {#2}}}
\def\bzero{{\bf 0}}

\def\divv{\nab\cdot}
\def\divx{\nab_x\cdot}
\def\divtx{\nab_{t,x}\cdot}
\def\nabx{\nab_x}

\newcommand{\grad}{\nabla}
\newcommand{\curlt}{{\nabla \times}}
\newcommand{\gperp}{\nabla^{\perp}}
\newcommand{\gradt}{\nabla\cdot}

\def\forallqq{\quad\forall\,}
\def\aph{A^{1/2}}
\def\amh{A^{-1/2}}

\def\osc{{\rm osc \, }}

\def\Im{{\rm Im}}
\newcommand{\tr}{{\rm tr}}
\def\divvr{{\rm div}}
\def\curllr{{\rm curl}}
\def\curll{{\rm curl}}
\def\curl{{\bf curl}}
\newcommand{\bgrad}{{\bf grad}}
\newcommand\diam{\mathrm{diam\,}}
\renewcommand\Im{\mathrm{Im\,}}
\def\Span{\mbox{Span}}
\def\supp{\mbox{supp\,}}
\newcommand{\trace}{{\rm trace}}

\newcommand{\tri}{|\!|\!|}
\newcommand{\ljump}{\lbrack\!\lbrack}
\newcommand{\rjump}{\rbrack\!\rbrack}
\newcommand{\bdm}{\begin{displaymath}}
\newcommand{\edm}{\end{displaymath}}
\newcommand{\beq}{\begin{equation}}
\newcommand{\eeq}{\end{equation}}
\newcommand{\beqa}{\begin{eqnarray}}
\newcommand{\eeqa}{\end{eqnarray}}
\newcommand{\beqas}{\begin{eqnarray*}}
\newcommand{\eeqas}{\end{eqnarray*}}
\newcommand{\ul}{\underline}
\newcommand{\wh}{\widehat}
\newcommand{\la}{\langle}
\newcommand{\ra}{\rangle}

\newcommand{\Lt}{L^2(\Omega)}
\newcommand{\Lts}{L^2(\Omega)^2}
\newcommand{\Ltc}{L^2(\Omega)^3}
\newcommand{\Ho}{H^1(\Omega)}
\newcommand{\Hoh}{H^1(\wh{\Omega})}
\newcommand{\Hoi}{H^1(\Omega_i)}
\newcommand{\Hos}{H^1(\Omega)^2}
\newcommand{\Hoc}{H^1(\Omega)^3}
\newcommand{\Hoch}{H^1(\wh{\Omega})^3}
\newcommand{\Hoci}{H^1(\Omega_i)^3}
\newcommand{\Hoz}{H^1_0(\Omega)}
\newcommand{\Ht}{H^2(\Omega)}
\newcommand{\Hti}{H^2(\Omega_i)}
\newcommand{\Hts}{H^2(\Omega)^2}
\newcommand{\Htc}{H^2(\Omega)^3}
\newcommand{\Htz}{H^0(\Omega)}
\newcommand{\Hh}{H^{1/2}(\Gamma)}
\newcommand{\Hhi}{H^{1/2}(\Gamma_i)}
\newcommand{\Hmh}{H^{-1/2}(\Gamma)}
\newcommand{\Hdiv}{H(\divvr;\,\Omega)}
\newcommand{\Hdivh}{H(\divv;\,\wh \Omega)}
\newcommand{\hcurl}{H(\curl\,A;\,\Omega)}
\newcommand{\Hcurl}{H(\curll\,A;\,\Omega)}
\newcommand{\Hcrl}{H(\curll\,;\,\Omega)}
\newcommand{\hcrl}{H(\curl\,;\,\Omega)}
\newcommand{\Hcrlh}{H(\curll\,;\,\wh\Omega)}
\newcommand{\hcrlh}{H(\curl\,;\,\wh\Omega)}
\newcommand{\Wdiv}{\BW_0(\mbox{\divv}\,;\,\Omega)}
\newcommand{\Wcurl}{\BW_0(\mbox{\curl}\,A;\,\Omega)}
\newcommand{\WcrossV}{\BW \times V}

\def\calS{{\cal S}}
\def\cH{{\cal H}}
\def\ba{{\mathbf{a}}}
\def\cN{{\cal N}}  

\def\bE{{\bf E}}
\def\bS{{\bf S}}
\def\br{{\bf r}}
\def\bW{{\bf W}}
\def\bLambda{{\bf \Lambda}}

\def\zT{{z_{_{\cT}}}}
\def\vT{{v_{_{\cT}}}}
\def\uT{{u_{_{\cT}}}}

\newcommand{\dd}{\underline{{\mathbf d}}}
\newcommand{\C}{\rm I\kern-.5emC}
\newcommand{\R}{\rm I\kern-.19emR}
\newcommand{\W}{{\mathbf W}}
\def\3bar{{|\hspace{-.02in}|\hspace{-.02in}|}}
\newcommand{\A}{{\mathcal A}}

\newcommand{\aA}{{ \a_{F,_A}}}

\newcommand{\aH}{{ \a_{F,_H}}}

\newcommand{\lJump}{[\![}
\newcommand{\rJump}{]\!]}
\newcommand{\jump}[1]{[\![ #1]\!]}

\newcommand{\red}[1]{{\color{red} {#1} }}

\def\Xint#1{\mathchoice
{\XXint\displaystyle\textstyle{#1}}%
{\XXint\textstyle\scriptstyle{#1}}%
{\XXint\scriptstyle\scriptscriptstyle{#1}}%
{\XXint\scriptscriptstyle\scriptscriptstyle{#1}}%
\!\int}
\def\XXint#1#2#3{{\setbox0=\hbox{$#1{#2#3}{\int}$}
\vcenter{\hbox{$#2#3$}}\kern-.5\wd0}}
\def\ddashint{\Xint=}
\def\dashint{\Xint-}

\title {Robust and Local Optimal A Priori Error Estimates for Interface Problems with Low Regularity: \\
Mixed Finite Element Approximations}
\author{
Shun Zhang\thanks{Department of Mathematics, City University of Hong Kong, Hong Kong SAR, China, shun.zhang@cityu.edu.hk.
}}
\date{\today}
\maketitle

\begin{abstract}
For elliptic interface problems in two- and three-dimensions with a possible very low regularity, this paper establishes a priori error estimates for the Raviart-Thomas and Brezzi-Douglas-Marini mixed finite element approximations. These estimates are robust with respect to the diffusion coefficient and optimal with respect to the local regularity of the solution. Several versions of the robust best approximations of the flux and the potential approximations are obtained. These robust and local optimal a priori estimates provide guidance for constructing robust a posteriori error estimates and adaptive methods for the mixed approximations.
\end{abstract}

\section{Introduction}\label{intro}
\setcounter{equation}{0}

As a prototype of problems with interface singularities, this paper studies {\em a priori} error estimates of mixed finite element methods for the following interface problem (i.e., the diffusion problem with discontinuous coefficients):
\begin{equation}\label{scalar}
	-\nabla\cdot \,(\a(x)\nabla\, u) = f
 	\quad \mbox{in} \,\,\Omega
\end{equation}
with homogeneous Dirichlet boundary conditions (for simplicity)
\beq\label{bc1}
	u = 0 \quad \mbox{on } \p \O,
\eeq
where $\Omega$ is a bounded polygonal domain in $\R^d$ with $d=2$ or $3$; $f \in L^{2}(\O)$ is a given function; and diffusion coefficient $\a(x)$ is positive and piecewise constant with possible large jumps across subdomain boundaries (interfaces):
\[
	\a(x)=\a_i > 0\quad\mbox{in }\,\O_i
	\quad\mbox{for }\, i=1,\,...,\,n.
\]
Here, $\{\Omega_i\}_{i=1}^n$ is a partition of the domain $\O$ with $\O_i$ being an open polygonal domain. 
It is well known that the solution $u$ of problem (\ref{scalar}) belongs to $H^{1+s}(\O)$ with possibly very small $s> 0$, see for example Kellogg \cite{Kel:75}. But we should also note that even the global regularity is low, when a finite element mesh is given, the singularity or those elements whose solution having a large gradient often only appear bear some points, or along a curve. Thus it is a bad idea to use the global regularity and a global uniform mesh-size to do the a priori error estimate. 

In \cite{CHZ:17}, we introduced the idea of robust and local optimal a priori error estimate. The robustness means that the genetic constants appeared in the estimates are independent of the parameters of the equation, the coefficient $\a$ in our case. The local optimality means that in the error estimate, the upper bound is optimal with the regularity of each element and local mesh sizes, instead of using a global uniform mesh size and a global regularity.  

The local optimal and robust a priori error estimate is very important for the adaptive mesh refinement algorithm. Since that all mesh refinements algorithms are based on the so-called "error equi-distribution" principle \cite{NoVe:12}, that is, each element has an almost equal size of the error measured in an appropriate norm, we need to show this is possible via a priori error estimate. In some sense, if we have a known exact solution $u$ so that the a priori error bound can be computed exactly, we should be able to find an optimal mesh with a fixed number of degrees of freedom that each element has a very similar size of the error. Also, in the robust a posteriori error analysis, we always try to find an equivalence between some intrinsic norm of the error and a computable error estimator, the so call the reliability and efficiency bounds. When constructing the error estimator, it is essential to realize that the best the adaptive numerical method can get is restricted by the robust local a priori estimates with respect each elements. This is especially important for the mixed methods, since there are two unknowns, the flux and the potential, and there are various post-processing methods. It is important to find which is the right quantity and norm to estimate in the a posteriori error estimates.

The proof of local optimal and robust a priori error estimate often contains two parts: one is the {\bf robust best approximation} result (Cea's lemma type of result), which has its own importance; the other is the {\bf robust local approximation properties of the interpolation operator}.

Before we discuss the robust best approximation result and robust local interpolations results for the mixed approximations, we first discuss the corresponding results for the conforming, Crouzeix-Raviart nonconforming, and discontinuous Galerkin results of the interface problem.


For the interface problem (\ref{scalar}), the robust best approximation property is well known and it almost trivial for the $H^1$ conforming approximation:
$$
\|\a^{1/2}\nabla (u-u_k^c)\|_0 \leq \inf_{v_k^c \in V_k^c}\|\a^{1/2}\nabla (u-v_k^c)\|_0,
$$
where $V_k^c$ is the $k$-th degree $H^1_0$-conforming finite element space, and $u_k^c$ is the corresponding $H^1$ conforming approximation. 

On the other hand, the proofs of the robust best approximation for CR nonconforming and discontinuous Galerkin is not easy. In \cite{CHZ:17}, for the Croueix-Raviart nonconforming element approximation, we showed the robust best approximation property (the constant $C$ independent of $\a$ and mesh size):
$$
\|\a^{1/2}\nabla_h (u-u_1^{nc})\|_0 \leq C\left( \inf_{v_1^{nc} \in V_1^{nc}}\|\a^{1/2}\nabla_h (u-v_1^{nc})\|_0 +\osc_{\alpha,nc} \right),
$$
where $V_1^{nc}$ is the Crouzeix-Raviart non-conforming finite element space, and $u_1^{nc}$ is the corresponding non-conforming approximation, and $\osc_{\alpha,nc}$ is a robust oscillation term. Also in \cite{CHZ:17}, for the discontinuous Galerkin approximation, we showed the robust best approximation property (the constant $C$ independent of $\a$ and mesh size):

$$
\tri u-u_k^{dg}\tri_{dg} \leq C\left( \inf_{v_k^{dg} \in D_k}\tri u-v_k^{dg}\tri_{dg}  +\osc_{\alpha,dg} \right),
$$
where $D_k$ is the $k$-th degree discontinuous finite element space, and $u_k^{dg}$ is the corresponding discontinuous Galerkin approximation, $\tri \cdot\tri_{dg}$ is the $\a$-weighted $H^1$ discontinuous Galerkin norm, and $\osc_{\alpha,dg}$ is a robust oscillation term.

The local approximation properties of the interpolation operators for the DG space and Crouzeix-Raviart  is easy to show. For the conforming finite element approximation, there are two types of local interpolations: nodal interpolations which require high regularity of the solution, and the Scott-Zhang or Clement interpolations whose regularity requirement is very low. For the nodal interpolation, it is completely local in each  element, but the it need very high regularity to exist, especially in three dimensions. For the Scott-Zhang/Clement interpolations, since they are defined on a local patch, their local robustness depends on a non-realistic assumption, the quasi-monotonicity assumption, see \cite{DrSaWi:96, BeVe:00, CaZh:09, CHZ:17}. Thus, the existence of robust local optimal result for the conforming finite element approximation for the low regularity interface problem is still open.   

For the mixed methods, we have two unknowns, one is the flux $\bsigma$, and the other is the potential $u$. For the potential $u$, the discontinuous finite element approximation is used, so the robust local interpolation property is obvious. We use Raviart-Thomas or Brezzi-Douglas-Marini elements to approximate to the flux variable, a robust local interpolation property can be proved by the average Taylor series technique developed in  \cite{DuSc:80}. This leaves the main task of proving the robust local optimal error estimates to the proof of the robust best approximation properties of the mixed methods. Unlike the conforming, non-conforming, or DG methods, we have several choices of the norms and the approximations spaces.


Our first robust best approximation property is simple, the weighted $L^2$-norm of the flux error in the equilibrated discrete spaces, see Theorem 3.2 and 3.3. 

For the potential $u$, in the standard  analysis of the mixed method, the $L^2$ norm is used. It turns out that we have difficulties to have a robust inf-sup condition with the weighted $L^2$ norm for the discrete approximation $u_h$ and a modified $H(\divvr)$ norm. Thus, we use the $\a$- and mesh-dependent norms to do the robust analysis. The choice of norm for $u_h$ is a norm similar to the standard discontinuous Galerkin norm, that is, a weighted discrete $H^1$ norm. With this $\a$- and mesh-dependent norm analysis, we show robust best approximation result for the potential approximation in the $\a$-dependent discrete $H^1$ norm. But since the approximation space for the potential $u$ is not rich enough, the order of approximation of $u$ in the $\a$-dependent discrete $H^1$ norm is one or two orders lower than the flux approximation. This order discrepancy suggests that we should not try to do the robust estimate of the $\a$ weighted discrete $H^1$-norm of the potential approximation in the a posteriori error analysis, as stated the earlier discussion by Kim \cite{Kim:07}.

For the flux approximation, with the help of $\a$- and mesh-dependent analysis, we show the robust best approximation result in the non-equilibrated RT/BDM space with an $\a$ and $h$ weighted $H(\divvr)$ norm for the first time. The corresponding robust and local a priori error estimates are also given without order loss even for the BDM approximations.

Finally, since the discrete $H^1$ norm of the potential approximation $u_h$ is often of a lower order than the corresponding flux approximation, we use Stenberg's post-processing to recover a new approximation with a compatible polynomial degree. We show that for the recovered potential approximation, the robust local best approximation result is true and a robust local a priori error estimates of the same order as the flux approximation is obtained. We also prove a new trace inequality of the normal trace. We also point out in the paper that any recovery or post-processing should based on the flux approximation since it is more accurate.

There are many a priori estimates for mixed methods available. The standard analysis  can be found in the books and papers \cite{DR:82, BBF:13, RT:91, Ga:14}. In these analysis, $L^2$ or $H(\divvr)$ norms are used for the flux approximation and the $L^2$ norm is used for the potential approximation. No robust analysis is discussed in these papers or books. The mesh-dependent norm analysis can be found in \cite{BrVe:96,LS:06}, also, no robust analysis is discussed. In \cite{Voh:07,Voh:10,Kim:07}, many a priori and a posteriori error results are presented for the mixed methods, some are robust and some are non-robust.  No robust and local optimal estimates are discussed for mixed methods before.

The paper is organized as follows. Section 2 describes the mixed finite element methods for the model problem. Various robust best approximations results and robust and local a priori error estimates are presented in Section 3, including the robust best approximation results for the flux in the weighted $L^2$ norm in the discrete equilibrated space and in the weighted $H(\divvr)$ norm in the whole mixed approximations spaces, the robust best approximation result for the potential in weighted discrete $H^1$ norm. In Section 4, we discuss Stenberg's of post-processing and show its robust and local optimal a priori error estimates in each elements. In Section 8, we make some concluding remarks.

%
%
%

\section{Mixed Finite Element Methods}

Introducing the flux 
\[
	\bsigma = -\a(x)\nabla u, 
\]
the mixed variational formulation for the problem in (\ref{scalar}) and (\ref{bc1}) is to find
$(\bsigma,\,u)\in H(\divvr;\O)\times L^2(\O)$ such that
\begin{equation}\label{mixed}
	\left\{\begin{array}{lclll}
 		(\a^{-1}\bsigma,\,\btau)-(\divv \btau,\, u)&=&0 \quad & \forall\,\, \btau \in H(\divvr;\O),\\[2mm]
		(\divv \bsigma, \,v) &=& (f,\,v)&\forall \,\, v\in L^2(\O).
	\end{array}\right.
\end{equation}

Let $\cT=\{K\}$ be a regular triangulation of the domain
$\Omega$ (see, e.g., \cite{Cia:78, BrSc:08}). 
Denote by $h_K$ the diameter of the element $K$. Assume that
interfaces $\{\p\O_i\cap\p\O_j\,:\, i,j=1,\,...,\,n\}$ do not cut
through any element $K\in\cT$. For any element $K\in\cT$, 
denote by $P_k(K)$ the space of polynomials on $K$ with total degree less than or equal to $k$. 

Define the discontinuous piecewise polynomial space of degree $k$ by
$$
D_k = \{ v \in L^2(\O)\, :\, v|_K \in P_k \; \forall\, K\in\cT\}.
$$
Define the $H(\divvr)$ conforming Raviart-Thomas (RT) finite element space 
and Brezzi-Douglas-Marini (BDM) finite element space of order $k$ by
$$
RT_k = \{ \btau \in H(\divvr;\O)\, :\, \btau|_K \in P_k(K)^d + \bx P_k(K) \; \forall\, K\in\cT\}.
$$
and 
$$
BDM_k = \{ \btau \in H(\divvr;\O)\, :\, \btau|_K \in P_k(K)^d \; \forall\, K\in\cT\}.
$$
For mixed problems, $RT_k\times D_k$ and $BDM_{k+1}\times D_k$ are stable pairs. Thus, we use the notation $\Sigma_k$ to denote $RT_k$ or $BDM_{k+1}$. 

The mixed finite element approximation is to find $(\bsigma_h,\,u_h) \in \Sigma_k \times D_k$ such that
\begin{equation}\label{problem_mixed}
	\left\{\begin{array}{lclll}
 		(\a^{-1}\bsigma_h,\,\btau_h)-(\divv \btau_h,\, u_h)&=&0
 		\quad & \forall\,\, \btau_h \in \Sigma_k,\\[2mm]
 		(\divv \bsigma_h,\, v_h) &=& (f,\,v_h)&\forall \,\, v_h\in D_k.
	\end{array}\right.
\end{equation}
Difference between (\ref{mixed}) and (\ref{problem_mixed}) yields the following error equation:
\begin{equation}\label{erroreq_mixed}
	\left\{\begin{array}{lclll}
 	(\a^{-1}(\bsigma-\bsigma_h),\,\btau_h)-(\divv \btau_h,\, u-u_h)&=&0
 	\quad & \forall\,\, \btau_h \in \Sigma_k,\\[2mm]
 	(\divv (\bsigma-\bsigma_h),\, v_h) &=& 0&\forall \,\, v_h\in D_k.
\end{array}\right.
\end{equation}

%
%
%

\section{Robust and Local Optimal A Priori Error Estimates}
\setcounter{equation}{0}

\subsection{Mixed finite element interpolations and approximation properties}
For a fixed $r>0$, denote by $I^{rt,k}_{h}: \Hdiv \cap [H^r(\O)]^d \mapsto RT_k$ the standard $RT$ interpolation operator and  $I^{bdm,k}_{h}: \Hdiv \cap [H^r(\O)]^d \mapsto BDM_k$ the standard $BDM$ interpolation operator.  We have the following local approximation property: for $\btau \in H^{s_K}(K)$, $s_K >0$,
\begin{eqnarray} \label{rti}
\|\btau - I^{\Sigma,k}_{h} \btau\|_{0,K}
  &\leq& C h_K^{\min\{k+1,s_K\}} |\btau|_{\min\{k+1,s_K\},K} \quad\forall\,\, K\in \cT,
\end{eqnarray}
with $I^{\Sigma,k}_{h} = I^{rt,k}_{h}$ or $I^{bdm,k}_{h}$. The estimate in (\ref{rti}) is standard for $s_K\geq 1$ and can be proved by the average Taylor series developed in \cite{DuSc:80} and the standard reference element technique with Piola transformation for $0<s_K<1$. We also should notice that the interpolations and approximation properties are completely local.

Denote by $Q^k_{h}:  L^2 (\O) \mapsto D_k$ the $L^2$-projection onto $D_k$.  The following commutativity property is well-known:
\begin{eqnarray}\label{comm}
 \gradt (I^{rt,k}_{h}\,\btau)&=&Q^k_{h}\,\gradt\btau \qquad
 \forallqq\,\btau\in\Hdiv \cap H^r(\O)^d \,\mbox{ with }\, r>0, \\[2mm]
\label{comm_bdm}
 \gradt (I^{bdm,k}_{h}\,\btau)&=&Q^{k-1}_{h}\,\gradt\btau \qquad
 \forallqq\,\btau\in\Hdiv \cap H^r(\O)^d \,\mbox{ with }\, r>0.
\end{eqnarray}

\begin{rem}
The requirement $r>0$ in $\Hdiv \cap [H^r(\O)]^d$ is to make sure that the mixed interpolations are well defined. Another choice is $\{\btau\in L^p(\O)^d\mbox{  and  }\gradt \btau \in L^2(\O)\}$ for $p>2$ or $W^{1,t}(K)$ for $t>2d/(d+2)$ as in  \cite{BBF:13}.  We use the Hilbert space based choice since it is more suitable for our analysis.
\end{rem}

\subsection{Robust best approximation in the discrete equilibrated space for the flux}
Define the discrete equilibrated space
$$
\Sigma_k^f = \{\btau_h \in \Sigma_k : \gradt \btau_h =Q^k_{h} f\}.
$$
Note that $\Sigma_k^f = RT_k^f = \{\btau_h \in RT_k : \gradt \btau_h =Q^k_{h} f\}$ for the RT case and $\Sigma_k^f = BDM_{k+1}^f= \{\btau_h \in BDM_{k+1} : \gradt \btau_h =Q^k_{h} f\}$ for the BDM case.

The following theorem is almost standard in the mixed finite element analysis.
\begin{thm}\label{apriori_mixed} (Robust best approximation in the discrete equilibrated space)
Let $(\bsigma, u)$ and $(\bsigma_h,\,u_h) \in \Sigma_k \times D_k$ be the solutions of {\em (\ref{mixed})} and {\em (\ref{problem_mixed})}, respectively, then the following robust best approximation result holds:
\begin{equation}\label{rba_equ}
\|\a^{-1/2}(\bsigma -\bsigma_h)\|_{0,\O} \leq \inf_{\btau_h^f \in \Sigma_k^f} \|\a^{-1/2}(\bsigma-\btau_h^f)\|_{0,\O}.
\end{equation}
\end{thm}
\begin{proof}
To establish (\ref{rba_equ}), denote by 
 \[
  \bE = \bsigma -\bsigma_h
   \quad\mbox{and}\quad
   e = u- u_h
 \]
the respective errors of the flux and the solution.

Now, let $\btau_h^f$ be an arbitrary function in $RT_k^f$, then it follows from the first equation in (\ref{erroreq_mixed}), the fact $\bsigma_h \in \Sigma_k^f$, and the Cauchy-Schwarz inequality that
\begin{eqnarray*}
  \|\a^{-1/2}\bE\|_{0,\O}^2
  &= & (\a^{-1}\bE,\, \bsigma-\btau_h^f) + (\a^{-1}\bE,\, \btau_h^f -\bsigma_h)\\
  &=&(\a^{-1}\bE,\, \bsigma-\btau_h^f) + (\divv (\btau_h^f-\bsigma_h),\,e)\\
  &=&(\a^{-1}\bE,\, \bsigma-\btau_h^f)
  \leq \|\a^{-1/2}\bE\|_{0,\O}\,\|\a^{-1/2}(\bsigma-\btau_h^f)\|_{0,\O},
\end{eqnarray*}
which implies the result of the theorem.
\end{proof}

\begin{thm}\label{apriori_mixed2} (Robust local a priori error estimates)
Let $(\bsigma, u)$ and $(\bsigma_h,\,u_h) \in \Sigma_k \times D_k$ $(k\geq 0)$ be the solutions of {\em (\ref{mixed})} and {\em (\ref{problem_mixed})}, respectively. Assume that $u\in H^{1+r}(\O)$ with some $r>0$ and that $u|_K\in H ^{1+s_K}(K)$ with an element-wisely defined regularity $s_K>0$ for all $K\in\cT$. Then there exists a constant $C>0$ independent $\a$ and $h$ for both the two- and three-dimension such that
\begin{eqnarray}\label{err-bound-L2RT}
\|\a^{-1/2}(\bsigma -\bsigma_h)\|_{0} &\leq& C \sum_{K\in\cT} h_K^{\min\{k+1,s_K\}} |\a^{1/2}\nabla u|_{\min\{k+1,s_K\},K},  \quad RT_k \mbox{  case} ,
	\\[2mm] \label{err-bound-L2BDM}
\|\a^{-1/2}(\bsigma -\bsigma_h)\|_{0} &\leq& C \sum_{K\in\cT} h_K^{\min\{k+2,s_K\}} |\a^{1/2}\nabla u|_{\min\{k+2,s_K\},K}, \quad BDM_{k+1} \mbox{  case}.
\end{eqnarray}
\end{thm}

\begin{proof}
For the $RT_k \times D_k$ case, the commutativity property in (\ref{comm}) and the second equations in (\ref{mixed}) and (\ref{problem_mixed}) lead to
$$
 \gradt (I_h^{rt,k}\bsigma) = Q^k_{h}\,\gradt\bsigma = Q^k_{h} f = \divv
 \bsigma_h.
$$
Thus, the result is a direct consequence of the best approximation property in  (\ref{rba_equ}) and the local approximation property in (\ref{rti}) by choosing $\btau_h^f = I_h^{rt,k}\bsigma \in RT_k^f$.

Using the same argument, we can get the result for the $DBM_{k+1} \times D_k$ case. 
\end{proof}

\begin{rem}
For those elements with a low regularity $0<s_K<1$, $RT_0$ is enough and there is no need to use BDM or high order RT approximations.
\end{rem}

\begin{rem}
For the case that in each element $K\in \cT$,  the diffusion coefficient being a full symmetric positive definite constant matrix $A|_K$ instead of a scalar constant $\a_K$, from the proofs, it is clear the above robust best approximation result is also true:
$$
\|A^{-1/2}(\bsigma -\bsigma_h)\|_{0,\O} \leq \inf_{\btau_h^f \in \Sigma_k^f} \|A^{-1/2}(\bsigma-\btau_h^f)\|_{0,\O}.
$$
In each element $K\in\cT$, for the quantity $\bq \in P_k^d$, $A^{-1/2} \bq$ is also in $P_k^d$, and thus $A^{-1/2}I_h^{\Sigma,k} \bq = A^{-1/2}\bq$. Thus for piecewise constant symmetric positive definite constant matrix $A$, we have
$$
\|A^{-1/2}(\btau - I^{\Sigma,k}_{h} \btau)\|_{0,K} 
	\leq C h_K^{\min\{k+1,s_K\}} |A^{-1/2}\btau|_{\min\{k+1,s_K\},K} \quad\forall\,\, K\in \cT.
$$
And we have the robust local a priori error estimatesL
\begin{eqnarray*}
\|A^{-1/2}(\bsigma -\bsigma_h)\|_{0,\O} &\leq& C \sum_{K\in\cT} h_K^{\min\{k+1,s_K\}} |A^{1/2}\nabla u|_{\min\{k+1,s_K\},K}, \quad RT_k \mbox{  case},
	\\[2mm]
\|A^{-1/2}(\bsigma -\bsigma_h)\|_{0,\O} &\leq& C\sum_{K\in\cT} h_K^{\min\{k+2,s_K\}} |A^{1/2}\nabla u|_{\min\{k+2,s_K\},K}, \quad BDM_{k+1} \mbox{  case}.
\end{eqnarray*}
The corresponding results for discontinuous Galerkin methods are not proved, since the robustness of the DG method for the diffusion problem depends on the right choice of the weights of the averages and penalty coefficients. For the full tensor case, the right weight is not clear or probably not possible for a full matrix $A$, see \cite{CHZ:17}. For the conforming finite element approximations, due to the lack of the nodal interpolations for the low regularity cases, such robust local optimal estimates is not available. For averaging operators like the Scott-Zhang or Clement interpolations, the robustness with respect to the full tensor $A$ is also impossible since even the famous quasi-monotonicity assumption is not meaningful in the case. For the Crouzeix-Raviart non-conforming finite element approximation, it is possible we can get a similar result by using the relation between the $RT_0$ and Crouzeix-Raviart elements.
\end{rem}
%
%

\subsection{Mesh-dependent norm analysis}
In this subsection, we use mesh-dependent norm analysis to derive the robust best approximation properties for the flux and the potential in appropriate norms. Earlier analysis on the mixed methods using mesh-dependent norms can be found in Babu\v{s}ka, Osborn, and Pitk\"{a}ranta \cite{BaOsPi:80}, Braess and Verf\"{u}rth \cite{BrVe:96}, and \cite{CaZh:12}. In the mesh-dependent analysis, we need to restrict ourselves to the scalar case.

First, we discuss the averages of the coefficients on the edge/face $F\in \cE$. For $F = \p K_F^{+} \cap \p K_F^{-}\in \cE_{I}$, denote by $\a^+_{F}$ and $\a^-_{F}$ the restriction of $\a$ on the respective $K_F^{+}$ and $K_F^{-}$. Denote the harmonic averages of $\a$ on $F \in \cE$ by
\[
 \a_{F,H} = \left\{\begin{array}{cl}
  \dfrac{\a_F^+  \a_F^- }{\a_F^+ + \a_F^-},&\quad F \in \cE_{I},\\[4mm]
 \a_F^-   &\quad F \in \cE_{\sD}\cup\cE_{\sN},
 \end{array}\right.
\]
which is equivalent to the minimum of $\a$:
\beq\label{a-h}
\dfrac{1}{2}\min\{\a_F^+, \a_F^- \}\leq \a_{F,H} \leq \min\{\a_F^+, \a_F^- \} .
\eeq

\begin{lem} The bilinear form $(\gradt \btau, v)$ for $(\btau,v)\in H(\divvr;\O)\times L^2(\O)$ 
has the following representation:
\beq \label{rep}
	(\gradt\btau, v) 
 	  = -\sum_{K\in\cT} (\nabla v,\btau)_{K}  
		+ \sum_{F\in \cE_{I}} (\btau\cdot\bn, \jump{v})_F 
		+ \sum_{F\in \cE_D} (\btau\cdot\bn, v)_F 
\eeq
\end{lem}

\begin{proof} The representation (\ref{rep}) is a consequence of integration by parts.
\end{proof}
Define $(\a,\,h)$-dependent norms on $\cT$ by
\begin{eqnarray*}
 && \| \btau \|_{\a,h}^2 :=\|\alpha^{-1/2} \btau\|_{0}^2 
	+ \dsum_{F\in \cE }\frac{h_F}{\a_{F,H}} \|\btau \cdot \bn\|_{0,F}^2, 
	\quad \forall 
	\btau \in \Sigma_k
 \\[2mm]
 	\mbox{and }\,\,&&
		\tri v\tri_{\a, h}^2  
		=\|\a^{1/2} \nabla_h v\|_{0,\cT}^2 
		+ \dsum_{F\in  \cE_{I}} \dfrac{\a_{F,H}}{h_F} \|\jump{v}\|_{0,F}^2
		+ \dsum_{F\in  \cE_{D}} \dfrac{\a_{F}}{h_F} \|v\|_{0,F}^2,
		\quad \forall v \in D_k.
\end{eqnarray*}
Note that the $\tri\cdot\tri_{\a,h}$ norm is the standard $\a$-weighted DG norm used in the discontinuous Galerkin methods, see \cite{CHZ:17}. For a $v\in H_0^1(\O)$,  $\tri v\tri_{\a, h} = \|\a^{1/2}\nabla v\|_{0,\O}$.

\begin{lem} \label{lem_hnorm}
For all $\btau \in \Sigma_k(K)$, there exists a positive constant $C>0$ independent of $\a$ and $h$, such that
\[
 \dsum_{F\in \cE_K}\frac{h_F}{\a_K} \|\btau \cdot \bn\|_{0,F}^2 \leq C \|\a^{-1/2}\btau\|_{0,K}^2.
\]
\end{lem}

\begin{proof}
The lemma is a simple consequence of the standard scaling argument and the fact that both $RT_k(K)$ and $BDM_{k+1}(K)$ are finite dimensional.
\end{proof}

\begin{thm} \label{thm_hnorm}
The following norm equivalence holds with $C>0$ independent of $\a$ and $h$:
\beq \label{norm_equ}
	\|\a^{-1/2}\btau_h\|_0 \leq \|\btau_h\|_{\a,h} 
	\leq C \|\a^{-1/2}\btau_h\|_0, \quad \forall \btau_h \in \Sigma_k.
\eeq
\end{thm}
\begin{proof}
Since for the harmonic average $\a_{F,H}$, we have $1/\a_{F,H} = 1/\a_{F}^+ +1/\a_{F}^-$,  by Lemma \ref{lem_hnorm}, we immediately get the robust discrete norm equivalence.
\end{proof}
For $\btau \in H(\divvr;\O)$, define the following $\a$ and $h$ dependent norm:
\beq
	\|\btau\|_{\a,h,H(\divvr)}:= \left(
	\|\a^{-1/2} \btau\|_0^2 + \sum_{K\in\cT}h_K^2\|\a^{-1/2}\gradt \btau\|_{0,K}^2 
	 \right)^{1/2}.
\eeq
We also use $\|\btau\|_{\a,h,H(\divvr),K}$ to denote the norm on a single element $K$.

The following trace inequality can be found in Lemma 2.4 and Remark 2.5 of \cite{CHZ:17}.
\begin{lem} Let $F$ be an edge/face of $K\in\cT$ and $\bn_F$ the unit vector normal to $F$. Assume that $\btau$ is a given function in $H(\divvr;K)\cap [H^r(K)]^d$, $r>0$ then for any $w_h\in P_k(K)$, we have 
\begin{eqnarray}\label{tracecombined}
 (\btau\cdot\bn, w_h)_F 
  &\leq & C\, h_F^{-1/2}\|w_h\|_{0,F}
 \left(\|\btau\|_{0,K} + h_K\|\gradt \btau\|_{0,K}\right).
\end{eqnarray}
\end{lem}
The following two continuity results are true. 
\begin{lem}
The following continuity results hold with constants $C_{con,1}>0$ and $C_{con,2}>0$ independent of $\a$ and $h$:
\begin{eqnarray} \label{con1}
(\gradt \btau_h,v) 
&\leq& C_{con,1}\|\a^{-1/2}\btau_h\|_0\tri v\tri_{\a,h}
, \quad
\forall \btau_h \in \Sigma_k, \quad v \in H^1_0(\O)
 \mbox{  or  } v\in D_k,\\[2mm]
 \label{cont_hDiv}
(\gradt \btau,v_h) 
&\leq & C_{con,2} \|\btau\|_{\a,h,H(\divvr)}\tri v_h\tri_{\a,h}, \quad
\forall \btau \in H(\divvr;\O)\cap [H^r(\O)]^d, \quad v \in D_k.
\end{eqnarray}
\end{lem}

\begin{proof}
The continuity (\ref{con1}) is clear from the representation (\ref{rep}), Cauchy-Schwarz inequality, 
the definition of norms $\|\btau\|_{\a,h}$ and $\tri v\tri_{\a,h}$, and the robust norm equivalent result (\ref{norm_equ}).

To show (\ref{cont_hDiv}), we still start from the representation (\ref{rep}):
$$ 
 (\gradt\btau, v_h) 
 	  = -\sum_{K\in\cT} (\nabla v_h,\btau)_{K}  
	+ \sum_{F\in \cE_{I}} (\btau\cdot\bn, \jump{v_h})_F 
	+ \sum_{F\in \cE_D} (\btau\cdot\bn, v_h)_F. 
$$
For the term $(\btau\cdot\bn, \jump{v_h})_F$, where $F\in\cE_I$, by (\ref{tracecombined}), 
\begin{eqnarray*}
  (\btau\cdot\bn,\,\, \jump{v_h})_F
  &\leq & C\, h_F^{-1/2}\|\jump{v_h}\|_{0,F}
 \left(\|\btau\|_{0,K} + h_K\|\gradt \btau\|_{0,K}\right),
\end{eqnarray*}
where $K$ is one of the elements having $F$ as an edge/face.
Choosing $K$ to be the element with the smaller $\a_K$. 
From (\ref{a-h}), the smaller $\a_K$ is equivalent to the harmonic average $\a_{F,H}$, then
\begin{eqnarray*}
  (\btau\cdot\bn,\,\, \jump{v_h})_F
  &\leq & C\, \a_{F,H}^{1/2}h_F^{-1/2}\|\jump{v_h}\|_{0,F}
 \left(\|\a^{-1/2}\btau\|_{0,K} + h_K\|\a^{-1/2}\gradt \btau\|_{0,K}\right).
\end{eqnarray*}
The term $(\btau\cdot\bn, v_h)_F$, $F\in\cE_D$, can be handled similarly. 
Then by the Cauchy-Schwarz inequality, (\ref{cont_hDiv})  can be easily proved.
\end{proof}

\begin{lem}
The following discrete inf-sup condition
\beq \label{infsup}
	\sup_{\btau_h \in \Sigma_k} \dfrac{(\gradt \btau_h,v_h)}{ \|\a^{-1/2}\btau_h\|_0} 
	\geq \beta \tri v_h \tri_{\a, h} \quad \forall\, v_h \in D_k
\eeq
holds with a constant $\beta>0$ independent of $\a$ and $h$.
\end{lem}

\begin{proof}
By the robust norm equivalent result (\ref{norm_equ}), we only need to prove the result for $\btau_h$ in the norm $\|\btau_h\|_{\a,h}$. Since $RT_k \subset BDM_{k+1}$, thus
$$
\sup_{\btau \in BDM_{k+1}} \dfrac{(\gradt \btau_h,v_h)}{ \|\btau_h\|_{\a,h}} 
\geq
\sup_{\btau \in RT_k} \dfrac{(\gradt \btau_h,v_h)}{ \|\btau\|_{\a,h}}, \quad  \forall\, v \in D_k,
$$
we only need to prove the RT version.
 
Choose a $\tilde{\btau}_h\in RT_k$ such that
\[
(\tilde{\btau}_h,\nabla q)_K = -(\a \nabla v, \nabla q)_K \quad \forall\,  q\in P_{k-1}(K)
\quad\forall\,\, K\in\cT
\]
and that  
\beq\label{n-bc}
\tilde{\btau}_h \cdot \bn |_F
 	=\left\{\begin{array}{llll}
 		\dfrac{\a_{F,H}}{h_F}\jump{v} 
	& \,\, F\in \cE_{I}, \\[3mm]
 	 	\dfrac{\a_F}{ h_F} v 	& \,\, F\in \cE_D,
	 \end{array}\right.
\eeq
which, together with (\ref{rep}), gives
\beq\label{5.10}
	(\gradt \tilde{\btau}_h,v_h)= \tri v \tri_{\a,h}^2.
\eeq

For every $K\in\cT$, by the standard scaling argument, there exists a constant $C>0$ independent of $\a$ and the mesh size such that
\[
\|\tilde{\btau}_h\|_{0,K}^2 \leq C \left (
\|\a_K \nabla v \|_{0,K}^2 + h_K \sum_{F\in \cE_K\cap\cE_{I}} 
\|\dfrac{\a_{F,H}}{h_F}\jump{v}\|_{0,F}^2
+h_K \sum_{F\in \cE_K\cap\cE_{D}} 
\|\dfrac{\a_{F}}{h_F} v\|_{0,F}^2
\right),
\]
which, together with (\ref{a-h}), gives
\[
\|\a_K^{-1/2}\tilde{\btau}_h\|_{0,K}^2 \leq C \left (
\|\a_K^{1/2} \nabla v \|_{0,K}^2 +\sum_{F\in \cE_K\cap\cE_{I}} \dfrac{\a_{F,H}}{h_F}
\|\jump{v}\|_{0,F}^2
+ \sum_{F\in \cE_K\cap\cE_{D}} 
\dfrac{\a_{F}}{h_F} \|v\|_{0,F}^2
\right),
\]
Hence, there exists a constant $\tilde{C}>0$ independent of $\a$ and $h$ such that
\[
	\|\tilde{\btau}_h\|_{\a,h} \leq \tilde{C} \tri v \tri_{\a,h}. 
\]
which, together with (\ref{5.10}), leads to the discrete inf-sup condition of the lemma.
\end{proof}

Define the following discrete divergence-free subspace of $\Sigma_k$:
$$
	\Sigma_k^0 =\{\btau_h \in \Sigma_k : \gradt \btau_h =0\}.
$$
Its orthogonal complement is 
$$
(\Sigma_k^0)^\perp =\{\btau_h \in \Sigma_k : (\btau_h, \brho_h) =0, \forall \brho_h \in \Sigma_k^0\}.
$$
Note that the inf-sup condition (\ref{infsup}) is also equivalent to the following inf-sup condition with $\beta>0$ independent of $\a$ and $h$:
\beq \label{infsup2}
	\sup_{v_h\in D_k} \dfrac{(\gradt \btau_h,v_h)}{ \tri v_h\tri_{\a,h}} 
	\geq \beta \|\btau_h\|_{\a,h} \geq \beta \|\a^{-1/2}\btau_h\|_0 \quad \forall\, \btau_h \in (\Sigma_k^0)^\perp.
\eeq
The condition (\ref{infsup}) also  guarantees that for each $g\in L^2(\O)$, there exists a unique solution $\btau_h \in (\Sigma_k^0)^\perp$ such that
\beq \label{orth}
	(\gradt \btau_h, v_h) = (g, v_h), \quad \forall v_h \in D_k.
\eeq
Now let us prove the following robust best approximation property for $\tri u-u_h\tri_{\a,h}$. 
\begin{thm} (Robust best approximation in the weighted discrete $H^1$ norm)
Let $(\bsigma, u)$ and $(\bsigma_h,\,u_h)\in \Sigma_k\times D_k$ 
be the solutions of 
{\em (\ref{mixed})} and {\em (\ref{problem_mixed})}, respectively.
Assume that $u\in H^{1+r}(\O)$ with $r>0$ and that $u|_K\in H ^{1+s_K}(K)$ with element-wisely defined $s_K>0$ for all $K\in\cT$. Then there exists a constant $C>0$ independent of $\a$ and $h$ for both the two- and three-dimension such that
\beq \label{aprioriu}
\tri u-u_h\tri_{\a,h} 
	\leq C\left( \inf_{\btau_h^f\in\Sigma_k^f} \|\a^{-1/2}(\bsigma-\btau_h^f)\|_{0,\O}+\inf_{v_h \in D_k}\tri u -v_h\tri_{\a,h}\right). 
\eeq
\end{thm}

\begin{proof}
By the inf-sup condition, for each $v_h \in D_k$ we have
\beq\label{uhvh}
\tri u_h - v_h\tri_{\a,h} \leq \frac{1}{\beta} \sup_{\btau_h \in \Sigma_k} \dfrac{(\gradt \btau_h, u_h-v_h)}{\|\a^{-1/2}\btau_h\|_0}.
\eeq
By the first equation in the error equations (\ref{erroreq_mixed}),
$$
(\gradt \btau_h, u_h-v_h) = (\gradt \btau_h, u-v_h) + (\gradt \btau_h, u_h-u) =
 (\gradt \btau_h, u-v_h) - (\a^{-1}(\bsigma-\bsigma_h),\,\btau_h).
$$
Then, by the continuity result \eqref{cont_hDiv} and the Cauchy-Schwarz inequality,
$$
(\gradt \btau_h, u_h-v_h) \leq C \|\btau_h\|_{\a,h} \tri u-v_h\tri_{\a,h} + \|\a^{-1/2}(\bsigma-\bsigma_h)\|_0 \|\a^{-1/2}\btau_h\|_{0}.
$$
Thus by \eqref{uhvh} and the equivalence of $ \|\btau_h\|_{\a,h}$ and $\|\a^{-1/2}\btau_h\|_{0}$,
$$
	\tri u_h - v_h\tri_{\a,h} \leq C(\tri u-v_h\tri_{\a,h}+ \|\a^{-1/2}(\bsigma-\bsigma_h)\|_0).
$$
A simple application of the triangle inequality yields
$$
\tri u-u_h\tri_{\a,h} \leq \tri u-v_h\tri_{\a,h}+\tri u_h-v_h\tri_{\a,h}
	\leq C\left(\|\a^{-1/2}(\bsigma-\bsigma_h)\|_0+ \tri u -v_h\tri_{\a,h}\right). 
$$
By the optimal convergence results of $\bsigma_h$, we have the robust best approximation result of the theorem.
\end{proof}

\begin{rem}
Even though we have the robust best approximation result (\ref{aprioriu}), due to the fact that the approximation orders of $\Sigma_k$ and $D_k$ are different for the corresponding norms, the order of convergence for $u-u_h$ in the discrete $H^1$ norm $\tri \cdot\tri_{\a,h}$ is one or two order lower than the corresponding weighted $L^2$ RT or BDM approximation errors in Theorem \ref{apriori_mixed2}, respectively. 

Due to this order difference, in the a posteriori error analysis, we should only construct the error estimator related to $\|\a^{-1/2}(\bsigma-\bsigma_h)\|_0$.
\end{rem}

Now, let us show the robust best approximation property in $\Sigma_k$.
\begin{thm} (Robust best approximation in the mixed approximation space)
The following robust best approximation properties are true with a constant $C$ independent of $\a$ and $h$:
\begin{eqnarray}
	\|\a^{-1/2}(\bsigma-\bsigma_h)\|_0 &\leq& C \inf_{\btau \in \Sigma_k}\|\bsigma- \btau_h\|_{\a,h,H(\divvr)}, \\[2mm] \label{rbahdiv}
	\|\bsigma- \bsigma_h\|_{\a,h,H(\divvr)} &\leq& C \inf_{\btau \in \Sigma_k}\|\bsigma- \btau_h\|_{\a,h,H(\divvr)}.
\end{eqnarray}
\end{thm}
\begin{proof}
For an arbitrary $\btau_h \in \Sigma_k$, by \eqref{orth}, there exists a unique $\bz_h \in (\Sigma_k^0)^\perp$, such that
$$
(\gradt \bz_h, v_h) = (\gradt (\bsigma- \btau_h), v_h), \quad\forall v_h \in D_k,
$$
and 
\beq
\beta \|\a^{-1/2}\bz_h\|_0 \leq \sup_{v_h \in D_k} \dfrac{(\gradt \bz_h, v_h)} {\tri v_h \tri_{\a,h}}=
\sup_{v_h \in D_k} \dfrac{(\gradt (\bsigma- \btau_h), v_h)} {\tri v_h \tri_{\a,h}}.
\eeq
By the continuity (\ref{cont_hDiv}), 
$$
(\gradt (\bsigma- \btau_h), v_h) \leq C\tri v_h \tri_{\a,h} \|\bsigma- \btau_h\|_{\a,h,H(\divvr)}.
$$
Thus, 
$$
\|\a^{-1/2}\bz_h\|_0  \leq C \|\bsigma- \btau_h\|_{\a,h,H(\divvr)}.
$$
Setting $\btau_h^f := \bz_h + \btau_h$, it is clear that $\btau_h^f \in \Sigma_k^f$. Then by the best approximation (\ref{rba_equ}),
$$
\|\a^{-1/2}(\bsigma-\bsigma_h)\|_0 \leq 
\|\a^{-1/2}(\bsigma-\btau_h^f)\|_0 \leq 
\|\a^{-1/2}(\bsigma-\btau_h)\|_0+\|\a^{-1/2}\bz_h\|_0 \leq C \|\bsigma- \btau_h\|_{\a,h,H(\divvr)}.
$$
On the other hand, since on each element $K\in \cT$, 
$$
(\gradt \bz_h, v_h)_K = (\gradt (\bsigma- \btau_h), v_h)_K, \quad\forall v_h \in P_k(K),
$$
and $\gradt \bz_h \in P_k(K)$, we have 
$$
	\|\gradt \bz_h\|_{0,K} \leq \|\gradt (\bsigma- \btau_h)\|_{0,K}.
$$
Since $\gradt(\bsigma_h-\btau_h^f)=0$, we have
\begin{eqnarray*}
\|\a^{-1/2}\gradt(\bsigma-\bsigma_h)\|_{0,K}&\leq&  
\|\a^{-1/2}\gradt(\bsigma-\btau_h^f)\|_{0,K}+\|\a^{-1/2}\gradt(\bsigma_h-\btau_h^f)\|_{0,K}\\
&=&
\|\a^{-1/2}\gradt(\bsigma-\btau_h^f)\|_{0,K}\\
&\leq&
 \|\a^{-1/2}\gradt(\bsigma-\btau_h)\|_{0,K}+
\|\a^{-1/2}\gradt \bz_h\|_{0,K} \\
&\leq & 2  \|\a^{-1/2}\gradt(\bsigma-\btau_h)\|_{0,K}.
\end{eqnarray*}
With this, the robust best approximation property (\ref{rbahdiv}) in $\|\cdot\|_{\a,h,H(\divvr)}$ can proved. 
\end{proof}
We classify the elements in the mesh into two sets:
\begin{eqnarray}
\cT_{low} = \{ K\in \cT : 0<s_K<1\} \quad\mbox{and}\quad
\cT_{high} = \{ K\in \cT : 1\leq s_K\}.
\end{eqnarray}

\begin{thm}\label{apriori_mixed3} (Robust local a priori error estimates in weighted $H(\divvr)$ norm)
Let $(\bsigma, u)$ and $(\bsigma_h,\,u_h) \in \Sigma_k \times D_k$ $(k\geq 0)$ be the solutions of {\em (\ref{mixed})} and {\em (\ref{problem_mixed})}, respectively. Assume that $u\in H^{1+r}(\O)$ with some $r>0$ and that $u|_K\in H ^{1+s_K}(K)$ with an element-wisely defined regularity $s_K>0$ for all $K\in\cT$. Then there exists a constant $C>0$ independent $\a$ and $h$ for both the two- and three-dimension such that
\begin{eqnarray}\label{err-bound-Div2RT}
\|\bsigma- \bsigma_h\|_{\a,h,H(\divvr)} &\leq& C
	\sum_{K\in\cT_{low}} \left(h_K^{s_K} |\a^{1/2}\nabla u|_{s_K,K} + h_K \|\a^{-1/2}f\|_{0,K}\right)\\
&&\quad	+ C\sum_{K\in\cT_{high}} \left ( h_K^{\min\{k+1,s_K\}} |\a^{1/2}\nabla u|_{\min\{k+1,s_K\},K}  \right. \\
&&\quad	 \left. + h_K^{\min\{k+2,s_K\}}\|\a^{-1/2}f\|_{\min\{k+1,s_K-1\},K}\right),  RT_{k} \mbox{  case}.
	\\[2mm] \label{err-bound-Div2BDM}
\|\bsigma- \bsigma_h\|_{\a,h,H(\divvr)} &\leq& C
	\sum_{K\in\cT_{low}} \left(h_K^{s_K} |\a^{1/2}\nabla u|_{s_K,K} + h_K \|\a^{-1/2}f\|_{0,K}\right)\\
&&\quad	+ C\sum_{K\in\cT_{high}} h_K^{\min\{k+2,s_K\}} \left ( |\a^{1/2}\nabla u|_{\min\{k+2,s_K\},K}  \right. \\
&&\quad	 \left. + \|\a^{-1/2}f\|_{\min\{k+1,s_K-1\},K}\right),  BDM_{k+1} \mbox{  case}.
\end{eqnarray}
\end{thm}
\begin{proof}
By the definition of the norm $\|\cdot\|_{\a,h,H(\divvr)}$, we only need to discuss the term 
$$
h_K\|\a^{-1/2}\gradt (\bsigma- \bsigma_h)\|_{0,K} = h_K\|\a^{-1/2}(f-Q^k_h f)\|_{0,K}
$$ 
for each element $K\in\cT$.

The first case is that the regularity is low in the element $K\in \cT_{low}$, with $0< s_K <1$. In this case, notice that $f \in L^2(K)$, thus
$$
h_K\|\a^{-1/2}(f-Q^k_h f)\|_{0,K} \leq h_K\|\a^{-1/2}f\|_{0,K}.
$$ 
Compared to the error $h_K^{s_K} |\a^{1/2}\nabla u|_{s_K,K}$ from the weighted $L^2$ approximation, it is of high order.

The other case is that $s_K \geq 1$ in the element $K$. Note that $\a_K$ is assumed to be a constant in $K$, thus $f = \gradt(\a_K\nabla u) = \a_K \Delta u \in H^{s_K-1}(K)$, thus
$$
h_K\|\a^{-1/2}(f-Q^k_h f)\|_{0,K} \leq C h_K^
{\min\{s_K,k+2\}}\|\a^{-1/2}f\|_{\min\{s_K-1,k+1\},K}.
$$ 
Compared with the weighted $L^2$ error, this term is of the same order for the $BDM_{k+1}$ approximation and one order high for the $RT_k$ approximation.
\end{proof}

\begin{rem}
One may want to use the Brezzi's theory directly as in \cite{LS:06} to get the following a priori error estimate
$$
\tri u - u_h\tri_{\a,h} +\|\bsigma - \bsigma_h\|_{\a, h} \leq C\left(\inf_{v\in D_k}\tri u - v_h\tri_{\a,h} +\inf_{\btau \in \Sigma_k}\|\bsigma - \btau_h\|_{\a, h}\right ).
$$
This is not right, since for problems with a low regularity, the $L^2$ norm of the trace $\|\bsigma\cdot\bn\|_{0,F}$ is not defined and thus $\|\bsigma \|_{\a,h}$ is not well-defined. Also, the result obtained by this is sub-optimal for the flux approximation.
\end{rem}

\begin{rem}
In the standard mixed method analysis, the $L^2$ norm of $u-u_h$ is analyzed and it has the same order convergence as the $RT$ approximation. In the case of the robust local a priori error estimate, we cannot get a robust local estimate for $\|\a^{1/2}(u-u_h)\|_0$ since robust an inf-sup condition
$$
\sup_{\btau_h \in \Sigma_k} \dfrac{(\gradt \btau_h,v_h)}{ \|\btau_h\|_{\a,h,H(\divvr)}} 
	\geq \beta \| \a^{1/2} v_h \|_0 \quad \forall\, v_h \in D_k,
$$
with a constant $\beta$ independent of $h$ and $\a$ is not available. 
\end{rem}

\section{Stenberg's Post-processing}
Since in the mixed methods, the approximation $u_h$ measured in the weighted discrete $H^1$ energy norm is lower than that of the approximation of the flux, we  introduce the Stenberg's post-processing to get a same order approximation.

On each element $K\in \cT$, if  $(\bsigma_h,u_h) \in RT_k\times D_k$ ($k\geq 0$) or $(\bsigma_h,u_h) \in BDM_k\times D_{k-1}$ ($k\geq 1$), i.e., the index of the flux approximation space is $k$, we find a $u_{h,K}^* \in P_{k+1}(K)$, such that
\beq
	(\a \nabla u_{h,K}^*, \nabla v_h)_K = (f,v_h)_K - (\bsigma_h\cdot\bn, v_h)_{\p K}, \quad \forall v_h\in P_{k+1}(K)/\R,
\eeq
and 
\beq
	\int_K u_{h,K}^* dx = \int_K u_{h} dx. 
\eeq

We first prove the following trace theorem by using techniques in \cite{BH:01,CaYeZh:11}.
\begin{thm}
For an element $K\in\cT$ with the mesh size $h_K$, we have
\beq \label{trace}
\|\btau\cdot\bn\|_{-1/2,\p K} \leq C(\|\btau\|_{0,K} + h_K \|\gradt \btau\|_{0,K}), \quad \forall \btau \in H(\divvr;K).
\eeq 
\end{thm}
\begin{proof}
For any $\btau\in H(\divvr;K)$ and $v\in H^1(K)$, we have the following identity:
\beq
\langle v, \btau\cdot\bn \rangle_{\p K} =(\btau, \nabla v)_K + (\gradt \btau, v)_K,
\eeq
where $\langle v, \btau\cdot\bn\rangle_{\p K}$ should be viewed as the duality pair between
$H^{1/2}(\p K)$ and $H^{-1/2}(\p K)$. 
Thus 
$$
\|\btau\cdot\bn\|_{-1/2,\p K} = \sup_{v\in H^{1/2}(\p K)} \dfrac
{(\btau, \nabla v)_K + (\gradt \btau, v)_K}{\|v\|_{1/2,\p K}}.
$$
On a reference element $\hat{K}$, given $g\in H^{1/2}(\p \hat{K})$, consider the following equation 
$$
- \Delta z + z =0 \in \hat{K}, \quad z = g \mbox{  on  } \p \hat{K}.
$$
By the elliptic stability theory, we have
$$
\|\nabla z\|_{0,\hat{K}} + \| z\|_{0,\hat{K}} \leq C\|g\|_{1/2, \p \hat{K}}.
$$
Mapping back to the physical element $K$ we have that given a $g\in H^{1/2}(\p K)$, 
there exits a $w_g\in H^1(K)$ and $w=g$ on $\p K$, such that 
$$
\|\nabla w_g\|_{0, K} + h_K^{-1} \| w_g\|_{0, K} \leq C\|g\|_{1/2, \p K}.
$$
Thus
$$
\|\btau\cdot\bn\|_{-1/2,\p K} 
\leq \dfrac
{(\btau, \nabla w_g)_K + (\gradt \btau, w_g)_K}{\|g\|_{1/2,\p K}}
\leq 
C(\|\btau\|_{0,K} + h_K \|\gradt \btau\|_{0,K}).
$$
The 
\end{proof}

\begin{thm} In each element $K\in\cT$, the following robust best approximation property holds:
\beq
\|\a^{1/2}_K\nabla(u-u_{h,K}^*)\|_{0,K} \leq C \left( \inf_{w_h \in P_{k+1}(K)}\|\a^{1/2}_K\nabla (u-w_h)\|_{0,K} + \|\bsigma-\bsigma_h\|_{\a,h, H(\divvr),K}
\right).
\eeq
\end{thm}
\begin{proof}
Let $w_h$ be an arbitrary function in $P_{k+1}(K)$, and $v_h= u_{h,K}^* - w_h$. Let $\overline{v}_h = \int_K v_h dx /|K|$ be the average of $v_h$ on $K$, then $v_h-\overline{v}_h$ belongs to the test space $P_{k+1}(K)/\R$. Then
\begin{eqnarray*}
\|\a^{1/2}_K\nabla(u_{h,K}^* - w_h)\|_{0,K}^2 
& = &\|\a^{1/2}_K\nabla v_h\|_{0,K}^2 = (\a\nabla(u_{h,K}^* - w_h), \nabla v_h)_K\\
& = &  (\a_K\nabla u_{h,K}^*,\nabla (v_h- \overline{v}_h))_K -(\a\nabla w_h, \nabla v_h)_K\\
& = & (f,v_h-\overline{v}_h)_K - (\bsigma_h\cdot\bn, v_h-\overline{v}_h)_{\p K} -(\a_K\nabla w_h, \nabla v_h)_K \\
& = & (\a_K \nabla (u-w_h),\nabla v_h)_K +((\bsigma-\bsigma_h)\cdot\bn, v_h-\overline{v}_h)_{\p K},
\end{eqnarray*}
where we use the fact that
$(\a_K \nabla u, \nabla v)_K = (f,v)_K - (\bsigma\cdot\bn, v)_{\p K}$ is true for any $v \in H^1(K)$.
 
By the Cauchy-Schwarz inequality, $(\a_K \nabla (u-w_h),\nabla v_h)_K \leq  \|\a_K^{1/2}\nabla (u-w_h)\|_{0,K}\|\a_K^{1/2}\nabla v_h\|_{0,K}$. By the definition of the dual norm, the trace inequality \eqref{trace}, and the fact $\|v_h-\overline{v}_h\|_{0,K} \leq C h_K\|\nabla v_h\|_{0,K}$, we have
\begin{eqnarray*}
	((\bsigma-\bsigma_h)\cdot\bn, v_h-\overline{v}_h)_{\p K} &\leq& \|\a^{-1/2}(\bsigma-\bsigma_h)\cdot\bn\|_{-1/2,\p K} \|\a^{1/2}(v_h-\overline{v}_h)\|_{1/2,\p K} \\
	&\leq & C h_K^{-1}\|\a^{1/2}(v_h-\overline{v}_h)\|_{0,K} \|\a^{-1/2}(\bsigma-\bsigma_h)\cdot\bn\|_{-1/2,\p K}\\
	&\leq & C \|\a^{1/2}\nabla v_h\|_{0,K}(\|\a^{-1/2}(\bsigma-\bsigma_h)\|_{0,K} + h_K\|\a^{-1/2}\gradt(\bsigma-\bsigma_h)\|_{0,K}).  
\end{eqnarray*}
Thus
$$
\|\a^{1/2}\nabla(u_{h,K}^* - w_h)\|_{0,K} \leq 
C(\|\a^{1/2}\nabla (u-w_h)\|_{0,K} + 
\|\a^{-1/2}(\bsigma-\bsigma_h)\|_{0,K} + h_K\|\a^{-1/2}\gradt(\bsigma-\bsigma_h)\|_{0,K}).
$$
By the triangle inequality, 
\beq
\|\a^{1/2}\nabla(u-u_{h,K}^*)\|_{0,K} \leq \|\a^{1/2}\nabla(u-w_h)\|_{0,K}+\|\a^{1/2}\nabla(u_{h,K}^* - w_h)\|_{0,K}.
\eeq
The theorem is proved.
\end{proof}

By the approximation property of $P_{k+1}(K)$, and the robust local optimal error estimate of $\bsigma_h$, we immediately have the following robust local optimal error estimate for the Stenberg's post-processing. 

\begin{thm}
For both the $(\bsigma_h,u_h) \in RT_k\times D_k$ ($k\geq 0$) or $(\bsigma_h,u_h) \in BDM_k\times D_{k-1}$ ($k\geq 1$) case, the Stenberg's recovery  $u_{h,K}^* \in P_{k+1}(K)$ has the following robust local a priori error estimate in the low regularity elements $K\in \cT_{low}$ with $0\leq s_K<1$:
\beq
\|\a^{1/2}_K\nabla(u-u_{h,K}^*)\|_{0,K}  \leq
C h_K^{s_K} |\a^{1/2}\nabla u|_{s_K,K} + h_K \|\a^{-1/2}f\|_{0,K}, K \in \cT_{low}.
\eeq
For those elements $K\in \cT_{high}$ with $1\leq s_K$, the following robust local a priori error estimate holds:
\begin{eqnarray}
\|\a^{1/2}_K\nabla(u-u_{h,K}^*)\|_{0,K}  &\leq& 
C  \left(  h_K^{\min\{k+1,s_K\}} |\a^{1/2}\nabla u|_{\min\{k+1,s_K\},K}  \right. \\
&& \left.+ h_K^{\min\{k+2,s_K\}}\|\a^{-1/2}f\|_{\min\{k+1,s_K-1\},K} \right), RT_{k}\times D_k \mbox{  case}.
	\\[2mm] 
\|\a^{1/2}_K\nabla(u-u_{h,K}^*)\|_{0,K} &\leq&  C h_K^{\min\{k+1,s_K\}} \left (|\a^{1/2}\nabla u|_{\min\{k+1,s_K\},K}  \right. \\
&&\quad	 \left. + \|\a^{-1/2}f\|_{\min\{k,s_K-1\},K}\right),  BDM_{k}\times D_{k
-1} \mbox{  case}.
\end{eqnarray}
\end{thm}

\begin{rem}
There are other post-processings available, such as the one proposed in \cite{AC:95} and analyzed in \cite{Voh:10}. The recovered potential is also mainly from the numerical flux $\bsigma_h$, a similar robust and local optimal a priori error estimate can also be derived.

It is also well known if the mixed method is implemented by hybridization, the Lagrange multiplier is also a better approximation of $u$ than $u_h$, and is a good source for post-processing or solution reconstruction. With careful analysis, it should not be hard to derive robust and local optimal result for the Lagrange multiplier and its post-processed solution under a similar weighted discrete $H^1$ norm.
\end{rem}

\section{Final comments}

In this paper, for elliptic interface problems in two- and three-dimensions with a possible very low regularity, we establish robust and local optimal a priori error estimates for the Raviart-Thomas and Brezzi-Douglas-Marini mixed finite element approximations. For the flux approximation, we show the robust best best approximation in the discrete equilibrated space and the whole mixed approximation space with appropriated norms, an $\a$-weighted $L^2$ norm or an $(\a,h)$-weighted $H(\divvr)$ norms. We show the robust local optimal error estimates for the flux approximation in these norms. For the potential approximation, we show a robust best approximation result in a weighted discrete $H^1$ norm and show that the convergence order is sub-optimal compared to the flux approximation. We then show that with the flux as the main source of post-processing, the Stenberg's post-processing can recover a potential with the robust local optimal error estimate.

 These robust and local optimal a priori estimates provide guidance for constructing robust a posteriori error estimates and adaptive methods for the mixed approximations. For robust a posteriori error for the mixed methods of the interface problem, we should focus on $\|\a^{-1/2}(\bsigma-\bsigma_h)\|_0$, like the approaches in \cite{Ain:07,CaZh:10,Kim:07,Voh:10}. The approaches in \cite{BrVe:96,LS:06} are not right since they are all try to put $u_h$ into the estimator. If any post-processing  is going to be used to construct the a posteriori error estimator, the main source of information should be the numerical flux $\bsigma_h$, not the numerical potential $u_h$ itself.

\end{document}